\documentclass[final]{dmtcs-episciences}
\usepackage{amsmath}
\usepackage{tikz}
\usetikzlibrary{matrix,arrows}
\usepackage{amsfonts,amssymb,amsthm, txfonts, pxfonts,amscd}
\usepackage[stretch=10]{microtype}
\usepackage{hyperref}
\def\struckint{\mathop{%
\def\mathpalette##1##2{\mathchoice{##1\displaystyle##2}%
 {##1\textstyle##2}{##1\scriptstyle##2}{##1\scriptscriptstyle##2}}%
\mathpalette
{\vbox\bgroup\baselineskip0pt\lineskiplimit-1000pt\lineskip-1000pt
\halign\bgroup\hfill$}
{##$\hfill\cr{\intop}\cr\diagup\cr\egroup\egroup}%
}\limits}

\def\sup{\mathop{\rm sup}\nolimits}

\def\symmetricn{\mathop{\mathcal{S}_{n}}\nolimits}
\def\symmetricntilde{\mathop{\overline{\mathcal{S}}_{n}}\nolimits}
\def\symmetricm{\mathop{\mathcal{S}_{m}}\nolimits}
\def\symmetricj{\mathop{\mathcal{S}_{j}}\nolimits}

\def\inv{\mathop{\text{inv}}\nolimits}

\def\part{\mathop{\mbox{Part}}\nolimits}

\def\ell{\mathop{[l]}\nolimits}

\def\red{\mathop{\text{red}}\nolimits}

\def\e{\mathbb{E}\,}
\def\R{\mathbb{R}\,}
\def\X{\mathbb{X}}

\def\B{\mathbb{B}}

\def\LOX{\mathcal{L}\left(\overline{\X}\right)}
\def\LOB{\mathcal{L}\left(\overline{\B}\right)}

\newcommand{\ignore}[1]{ }

\received{2017-3-22}

\revised{2017-9-26}

\accepted{2018-8-3}

\publicationdetails{19}{2018}{2}{13}{3213}

\newtheorem{thm}{Theorem}[section]
\newtheorem{lemma}[thm]{Lemma}
\newtheorem{prop}[thm]{Proposition}

\newtheorem{defn}[thm]{Definition}

\newtheorem{rmk}[thm]{Remark}
\title{Pattern avoidance for random permutations}
\author{Harry Crane\affiliationmark{1}\thanks{partially supported by NSF CAREER Award (DMS-1554092)}  
   \and  Stephen DeSalvo\affiliationmark{2}}
\affiliation{
Rutgers University, 110 Frelinghuysen Road, Piscataway, NJ, 08854 \\
University of California Los Angeles, 520 Portola Plaza, Los Angeles, CA, 90095}
\date{\today}


\begin{document}
\maketitle
\begin{abstract}
Using techniques from Poisson approximation, we prove explicit error bounds on the number of permutations that avoid any pattern. 
Most generally, we bound the total variation distance between the joint distribution of pattern occurrences and a corresponding joint distribution of independent Bernoulli random variables, which as a corollary yields a Poisson approximation for the distribution of the number of occurrences of any pattern.
We also investigate occurrences of consecutive patterns in random Mallows permutations, of which uniform random permutations are a special case.
These bounds allow us to estimate the probability that a pattern occurs any number of times and, in particular, the probability that a random permutation avoids a given pattern.
\end{abstract}

\setcounter{tocdepth}{1}
\tableofcontents

\section{Introduction}\label{section:intro}

A {\em permutation} of $[n]=\{1,\ldots,n\}$ is a bijection $\sigma:[n]\rightarrow[n]$, $i\mapsto\sigma(i)=\sigma_i$, written $\sigma=\sigma_1\cdots\sigma_n$.
For each $n=1,2,\ldots$, we write $\symmetricn$ to denote the set of permutations of $[n]$.
Given permutations $\sigma=\sigma_1\cdots\sigma_n$ and $\tau=\tau_1\cdots\tau_m$, we say that {\em $\sigma$ avoids $\tau$} if there does not exist a subsequence $1\leq i_1<\cdots <i_m\leq n$ such that $\sigma_{i_1}\cdots\sigma_{i_m}$ is order-isomorphic to $\tau$, and we say that {\em $\sigma$ avoids $\tau$ consecutively} if there is no $j=1,\ldots, n-m+1$ such that $\sigma_j\sigma_{j+1}\cdots\sigma_{j+m-1}$ is order-isomorphic to $\tau$.
Here we study pattern avoidance probabilities for random permutations from the Mallows distribution, which is of particular interest in the fields of statistics and probability but some special cases are insightful for questions in enumerative and extremal combinatorics.

To demonstrate our approach, let $n$ be any integer larger than or equal to 6.  It was shown in~\cite{pratt1973computing} that each fixed pattern $\tau$ of length $n-1$ is avoided by exactly $n! - (\, (n-1)^2+1\, )$ permutations in $\symmetricn$. 
In~\cite{ray2003posets}, it was shown that for each fixed pattern $\tau$ of length $N \equiv n-2$, there exists $0 \leq j \leq N-1$ such that exactly 
\[ n! - \frac{N^4 +2N^3+N^2+4N+4-2j}{2} \]
permutations in $\symmetricn$ avoid $\tau$. 
It is left as an open problem in~\cite{vatter376problems} to perform the same analysis for patterns of length $n-3$. 
From this, it is natural to ruminate on whether for each pattern $\tau$ of length $n-m$, one has approximately 
\begin{equation}\label{approx} n! - \frac{n!}{(n-m)!} \binom{n}{n-m}\end{equation}
permutations in $\symmetricn$ which avoid it, and how large $m$ can be taken. 
In this paper, we quantify the quality of approximation of~\eqref{approx} for counting the number of permutations in $\symmetricn$ which avoid any given pattern $\tau$ of length $n-m$, $1 \leq m \leq n$, via Poisson approximation. 
This approach yields general results which specialize to the counting problem just described, as well as the analogous problem for consecutive pattern avoidance. 

For classical pattern avoidance, we are able to show a good approximation if for any $\epsilon > 0$ we have $m \geq (e\, e^{1/e} + \epsilon)\sqrt{n}$, but we suspect the $e\, e^{1/e}$ constant can be improved with a natural lower bound of $e$ demonstrated by Lemma~\ref{lemma:d1bound}. 
For consecutive pattern avoidance, we obtain stronger results, namely, we may take any $m \geq \lceil \Gamma^{-1}(n)-1\rceil$.  (In the above expression, $\Gamma^{-1}(n)$ is the inverse of the gamma function, i.e., $\Gamma(t)=\int_0^{\infty}x^t e^{-x}dx$, which gives the usual generalization of the factorial function to all positive real numbers.)
The above expression is asymptotically best possible since, in the limit, taking $m$ equal to the right-hand side yields a Poisson approximation with nonzero rate parameter for the number of occurrences of a consecutive pattern.  
The range of values for $m$ follows by a detailed analysis of the bounds contained in Theorem~\ref{theorem:main} and Theorem~\ref{consecutive patterns}, respectively. 

The Mallows permutations we study here is a general class of random permutations whose distribution is weighted by the number of inversions.
An {\em inversion} in $\sigma=\sigma_1\cdots\sigma_n$ is a pair $(i,j)$, $i<j$, such that $\sigma_i>\sigma_j$.
For example, $\sigma=34125$ has four inversions, $(1,3), (1,4), (2,3), (2,4)$.
We write $\inv(\sigma)$ to denote the set of inversions of $\sigma$.
With $\Sigma_n$ denoting a random permutation of $[n]$, the {\em Mallows distribution with parameter $q\in(0,\infty)$} on $\symmetricn$ assigns probability
\begin{equation}\label{eq:Mallows}
P\{\Sigma_n=\sigma\}=P_n^q(\sigma)=q^{|\inv(\sigma)|}/I_n(q),\quad\sigma\in\symmetricn,\end{equation}
where $I_n(q)=\prod_{j=1}^n\sum_{i=0}^{j-1}q^i$ is the {\em inversion polynomial} and $|\inv(\sigma)|$ is the number of inversions in $\sigma$.
Note that $q=1$ corresponds to the uniform distribution on $\symmetricn$, i.e., $P\{\Sigma_n=\sigma\}=1/n!$ for all $\sigma\in\symmetricn$,  and is the critical point at which the Mallows family switches from penalizing inversions, $q<1$, to favoring them, $q>1$.

The Mallows distribution \cite{Mallows1957} was introduced as a one-parameter model for rankings that occur in statistical analysis.
More recently, Mallows permutations have been studied in the context of the longest increasing subsequence problem \cite{Bhatnagar2014} and quasi-exchangeable random sequences \cite{GnedinOlshanski2009,GnedinOlshanski2010}. 
For general values of $q>0$, we consider the problem of consecutive pattern avoidance for random Mallows permutations, with Theorem~\ref{main:mallows:theorem} establishing explicit error bounds on the entire distribution of the number of occurrences of patterns in a random permutation.
Our main theorems, therefore, complement prior work by Elizalde \& Noy \cite{elizaldenoy}, Perarnau \cite{perarnau}, and the more recent work by the current authors \& Elizalde ~\cite{crane2016probability} on consecutive pattern avoidance, as well as Nakamura \cite{Nakamura2014}, who used functional equations to enumerate sets with a prescribed number of occurrences of a given pattern. 

Our approach also differs from previous work  in a few key respects. 
While most prior work seeks either exact or asymptotic enumeration of the sets that avoid a given pattern or collection of patterns, we use the Chen--Stein Poisson approximation method \cite{chen1975poisson}, in particular \cite{ArratiaGoldstein}, to bound the total variation distance between the collection of all \emph{dependent} indicator random variables indicating pattern occurrence for a prescribed set of indices, and a joint distribution of \emph{independent} Bernoulli random variables with the same marginal distributions. 
These bounds allow us to approximate any measurable function of the occurrences, e.g., the number of patterns, locations of patterns, etc., via the corresponding independent random variables. 
We also reach a natural limit of the usefulness of these approximations corresponding to the strength of interactions between occurrences of patterns. 

The next section gives a brief historical context of restricted permutations. 
Section~\ref{section:pattern} presents the main results, of which there are two kinds: the first type is a preasymptotic bound given in terms of quantities which are complicated to compute but as accurate as the method allows; the second type is a more detailed analysis of the bounds in a form that is better suited to applications. 
Section~\ref{section:poisson} presents the Chen--Stein Poisson approximation approach that we utilize throughout the paper. 
In Section~\ref{section:Mallows} we apply this method to Mallows permutations. 
Section~\ref{section:cases} contains explicit numerical examples. 
Finally, the main technical results are proved in Section~\ref{section:proofs}.

\section{Motivation}\label{section_motivation}
Restricted permutations fall into two broad classes.  
The first, more tractable type is of the form $\sigma(a) \neq b$ for $a, b\in [n]$, whose study dates to the classical {\em probl\`emes des rencontres} in the early 1700s \cite{Montmort1708}; see also \cite[Chapter 4]{barbour1992poisson}.
A special case counts the number $D_n$ of {\em derangements} of $[n]$, i.e., permutations of $[n]$ without fixed points, for which we have the asymptotic expression
\begin{equation}\label{eq:derangement}D_n=n!\sum_{i=0}^n\frac{(-1)^i}{i!}\sim n!/e\quad\text{as }n\rightarrow\infty.\end{equation}
Equation \eqref{eq:derangement} can be stated in probabilistic terms by letting $\Sigma_n$ be a uniform random permutation of $[n]$, i.e., $P\{\Sigma_n=\sigma\}=1/n!$ for each $\sigma\in\symmetricn$, for which we compute
\begin{equation}\label{eq:derangement-prob}P\{\Sigma_n\text{ is a derangement}\}=D_n/n!\sim1/e\quad\text{as }n\rightarrow\infty.\end{equation}
See~\cite{ArratiaTavare,SheppLloyd} for more thorough treatments involving the cycle structure of random permutations.  

We can also derive the expression in \eqref{eq:derangement} by Poisson approximation. 
With $W$ denoting the number of fixed points in a random permutation of $[n]$, we demonstrate in Section~\ref{section:fixed}, see also \cite[Chapter 4]{barbour1992poisson}, that the distribution of $W$ converges in total variation distance to the distribution of an independent Poisson random variable with expected value $1$.  
In addition to the asymptotic value for the probability that a random permutation has no fixed points, this approach bounds the absolute error of probabilities that involve {any measurable function} of the number of fixed points in a random permutation.

The second type of restriction is {\em pattern avoidance}, which attracts increasing attention in the modern probability \cite{Bhatnagar2014,HoffmanRizzolo2016I,HoffmanRizzolo2016II} and modern combinatorics literature \cite{BonaBook}.
Any sequence of distinct positive integers $w=w_1\cdots w_k$ determines a permutation of $[k]$ by {\em reduction}: with $\{w_{(1)},\ldots,w_{(k)}\}_{<}$ denoting the set of elements listed in increasing order, we define the map $w_{(i)}\mapsto i$, under which $w$ maps to a permutation $\red(w)$ of $[k]$, called the {\em reduction} of $w$.
For example, $w=826315$ reduces to $\red(w)=625314$.
We call any fixed $\tau\in\symmetricm$ a {\em pattern} and say that $\sigma\in\symmetricn$ {\em contains} $\tau$ if there exists a subsequence $1\leq i_1<\cdots<i_m\leq m$ such that $\red(\sigma_{i_1}\cdots\sigma_{i_m})=\tau$.
We say $\sigma\in\symmetricn$ {\em avoids} $\tau$ if it does not contain it.
We say that $\sigma$ {\em contains $\tau$ consecutively} if there exists an index $j\in[n-m+1]$ such that $\red(\sigma_{j}\sigma_{j+1}\cdots\sigma_{j+m-1})=\tau$; otherwise, we say $\sigma$ {\em avoids $\tau$ consecutively}.
For any pattern $\tau$, we define 
\begin{align}
\symmetricn(\tau)&:=\{\sigma\in\symmetricn:\,\text{$\sigma$ avoids $\tau$}\}\quad\text{and}\notag\\
\label{eq:consecutive-set}
\symmetricntilde(\tau)&:=\{\sigma\in\symmetricn:\,\text{$\sigma$ avoids $\tau$ consecutively}\},
\end{align}
which we extend to any subset $A\subset\bigcup_{n\geq1}\symmetricn$ by
\begin{align}
\label{eq:avoiding-set}\symmetricn(A)&:=\{\sigma\in\symmetricn:\,\sigma\text{ avoids all }\tau\in A\}\quad\text{and}\\
\symmetricntilde(A)&:=\{\sigma\in\symmetricn:\,\sigma\text{ avoids all $\tau\in A$ consecutively}\}\notag.
\end{align}
For the most part, we are interested in sets $\symmetricn(\tau)$ containing all permutations that avoid a given pattern $\tau$, though our approach extends in a straightforward manner for more general sets $A$.

Much effort has been devoted to exact enumeration of $\symmetricn(A)$ for certain choices of $A$, see, e.g., \cite{TwoPatterns4,TwoPatterns1,TwoPatterns2,TwoPatterns3}.
But enumerating $\symmetricn(\tau)$ is notoriously difficult for patterns of fixed length larger than 3.  
Knuth \cite{KnuthI} initiated interest in pattern avoidance in the study of algorithms by identifying the $231$-avoiding permutations as exactly those that can be sorted by a single run through a stack; see Bona \cite[Chapter 8]{BonaBook} for further discussion.
In fact, it is now well known that the avoidance sets $\symmetricn(\tau)$ for every length-$3$ pattern $\tau$ are enumerated by the Catalan numbers \cite[A000108]{OEIS}:
\[|\symmetricn(\tau)|={2n\choose n}/(n+1),\quad \tau\in\{123,132,213,231,312,321\}.\]

Just as in the derangement problem above, enumeration of $\symmetricn(A)$ has an elementary probabilistic interpretation that motivates much of our paper.
With $\Sigma_n$ denoting a uniform random permutation of $[n]$ and $A$ a set of permutations, the probability that $\Sigma_n$ avoids $A$ is
\[P\{\Sigma_n\text{ avoids every }\tau\in A\}=|\symmetricn(A)|/n!.\]
The Stanley-Wilf conjecture as proved in~\cite{MarcusTardos2004} states that $|\symmetricn(\tau)|$ grows exponentially with $n$ for every {fixed} $\tau$.
For example, the Catalan numbers are known to grow asymptotically like $4^n / \sqrt{\pi\, n^3}$, yielding the asymptotic avoidance probability
\[P\{\Sigma_n\text{ avoids }231\}=\binom{2n}{n}/(n+1)!\sim \frac{1}{\pi\, n^2 \sqrt{2}}\left(\frac{4\, e}{n}\right)^n \quad\text{as }n\rightarrow\infty.\]
Such calculations quickly become intractable as $n$ grows large.
For example, the sets of $1324$-avoiding permutations have only been enumerated up to $n=31$ \cite{JohanssonNakamura2014}.
Even precise asymptotics for $|\symmetricn(1324)|$ have not yet been established \cite{Bona2014b,Bona2014a,ClaessonJelinek2012}.

\section{Main Results}\label{section:pattern}
To fix notation throughout the text, we write $\sigma=\sigma_1\cdots\sigma_n$ to denote a generic permutation.
For any subset $A\subseteq[n]$, we write $\sigma_{|A}$ to denote the {\em restriction} of $\sigma$ to a permutation of $A$ obtained by removing those elements among $\sigma_1,\ldots,\sigma_n$ that are not in $A$.
For example, with $\sigma=867531924$ and $A=\{1,3,5,7,9\}$, we have $\sigma_{|A}=75319$.
We write $\Sigma_n$ to denote a random permutation of $[n]$.

\subsection{Definitions}
\label{section:definitions}
Throughout the paper, we write $\mathcal{L}(X)$ to denote the {\em distribution}, or {\em law}, of a random variable $X$ and $\mathcal{L}(Y\mid X)$ to denote the {\em conditional distribution of $Y$ given $X$}.
For random variables $X$ and $Y$, we write $d_{TV}(\mathcal{L}(X),\mathcal{L}(Y))$ to denote the {\em total variation distance} between the distributions of $X$ and $Y$, which is defined as 
\[\displaystyle  d_{TV}(\mathcal{L}(X), \mathcal{L}(Y)) = \sup_{A \subseteq \mathbb{R}} \left| P(X \in A) - P(Y \in A)\right|, \]
where the $\sup$ is taken over Borel measurable subsets of $\mathbb{R}$. 
In the special case of non-negative integer-valued random variables, the total variation distance can be computed as
\[d_{TV}(\mathcal{L}(X),\mathcal{L}(Y))=\frac{1}{2}\sum_{n=0}^{\infty}|P(X=n)-P(Y=n)|.\]

Define the set of all unordered, distinct $m$-tuples of elements from $[n]$ by
\[ J_m := \{\{i_1, \ldots, i_m\} : 1\leq i_1<\cdots<i_m\leq n\}, \qquad m\in [n]. \]
For each $\alpha \in J_m$, let $D_\alpha$ be the set of all $m$-element subsets of $[n]$ that overlap with $\alpha$ in at least one element, i.e., $D_{\alpha}:=\{\beta\in J_m:\,\beta\cap\alpha\neq\emptyset\}.$
  For example, if $\alpha = \{1, 5, 8\}$, then $D_{\alpha}^c = \{ \{j_1, j_2, j_3\} : j_i \notin\{1,5,8\}, i=1,2,3\}.$  

With $\Sigma_n$ denoting a uniform random permutation of $[n]$, $\alpha \in  J_m$, and $\tau$ a fixed pattern of length $m$, we define $X_{\alpha}=X_{i_1, \ldots, i_m}$ as the indicator random variable for the event that the reduction of $\Sigma_n$ at positions $i_1, \ldots, i_m$ gives the pattern $\tau$, i.e.,
\begin{equation}\label{eq:indicator}
X_{i_1,\ldots,i_m}:=\mathbb{I}( { \red(\Sigma_n(i_1)\cdots\Sigma_n(i_m))=\tau } ).\end{equation} 
Let $\mathbb{X} \equiv \mathbb{X}_m := (X_\alpha)_{\alpha\in J_m}$ be the collection of all such indicators and 
let $\mathbb{B} \equiv \mathbb{B}_m = (B_\alpha)_{\alpha\in  J_m}$ denote a collection of independent Bernoulli random variables whose marginal distributions satisfy $\e B_\alpha = \e X_\alpha$ for each $\alpha \in  J_m$.  
The random variable
\begin{equation}\label{eq:W}
 W = \sum_{1 \leq i_1 < \ldots < i_m \leq n} X_{i_1,\ldots, i_m}
 \end{equation}
counts the total number of occurrences of $\tau$ in $\Sigma_n$.
For any $\tau\in\symmetricm$, for each $s=1,\ldots,m-1$ we define $L_s(\tau)$ as the {\em overlap of size $s$}, i.e., the number of permutations $\sigma\in\mathcal{S}_{2m-s}$ for which there are indices $1\leq i_1<\cdots<i_m\leq 2m-s$ and $1\leq j_1<\cdots<j_m\leq 2m-s$ such that $\{i_1,\ldots,i_m\}$ and $\{j_1,\ldots,j_m\}$ have exactly $s$ elements in common and 
\[\red(\sigma_{i_1}\cdots\sigma_{i_m})=\red(\sigma_{j_1}\cdots\sigma_{j_m})=\tau.\]

For consecutive pattern avoidance, we similarly define the set of all $m$-tuples of the form $\{i, i+1,\ldots, i+m-1\}$, $1 \leq i\leq n-m+1$, as
\[ \overline{J}_m:= \{\{i, i+1, \ldots, i+m-1\} : 1\leq i \leq n-m+1\}, \qquad m\in [n]. \]
Let $\overline{\mathbb{X}} \equiv \overline{\mathbb{X}}_m := (X_\alpha)_{\alpha\in\overline{J}_m}$, and 
let $\overline{\mathbb{B}} \equiv \overline{\mathbb{B}}_m = (B_\alpha)_{\alpha\in \overline{J}_m}$ denote a collection of independent Bernoulli random variables whose marginal distributions satisfy $\e B_\alpha = \e X_\alpha$ for each $\alpha \in \overline{J}_m$.  
For fixed $\tau\in\mathcal{S}_m$ and $\Sigma_n$ a uniform random permutation of $n$, we define the random variable
\begin{equation}\label{W'} \overline{W} := \sum_{1 \leq s \leq n-m+1} X_{s,s+1,\ldots,s+m-1},
\end{equation}
which counts the number of consecutive occurrences of $\tau$ in $\Sigma_n$.
We also define $\overline{L}_s(\tau)$ as the {\em sequential overlap of size $s$}, i.e., the number of permutations $\sigma\in\mathcal{S}_{2m-s}$ for which
\[\red(\sigma_{1}\cdots\sigma_{m})=\red(\sigma_{m-s+1}\cdots\sigma_{2m-s})=\tau.\]

\subsection{Main results on pattern avoidance}\label{main:corollaries}
We begin with several theorems specifically about pattern avoidance, which follow from the quantitative bounds given in Section~\ref{poisson limits}. 

\begin{thm}\label{theorem:asymptotic-general}
Assume $n \geq m \geq 3$ and $\tau$ is any pattern of length $m$.  Define
\begin{align}
\notag \lambda & = \binom{n}{m}/m!,  \qquad  d_1  = \binom{n}{m} \left( \binom{n}{m} - \binom{n-m}{m} \right) \frac{1}{m!^2}, \\
\label{d2} d_2 & = \sum_{s=1}^{m-1} \binom{n}{2m-s}\frac{2\, L_s(\tau)}{(2m-s)!}, 
\end{align}
and 
\[D_{n,m} = \min(1, \lambda^{-1})(d_1+d_2). \]
Then we have 
\begin{equation}\label{bounds}
n!\, e^{-\lambda}\left(1 - e^{\lambda} D_{n,m}\right) \leq |\symmetricn(\tau)| \leq n!\, e^{-\lambda}\left(1 + e^\lambda D_{n,m}\right).
\end{equation}

In addition, for any fixed $\epsilon\geq 0$, suppose $m\equiv m(n)$ is some increasing, integer-valued function of $n$ such that $m \geq (e\, e^{1/e}+\epsilon)\sqrt{n}$.
For any sequence of patterns $\tau_n\in\mathcal{S}_{m}$, we have
\[ \left| \frac{|\symmetricn(\tau)|}{n!} - e^{-\lambda}\right| = o(1)\quad\text{as }n\to\infty. \]
\end{thm}

\begin{rmk}
There are several noteworthy aspects to Theorem \ref{theorem:asymptotic-general}.  

\begin{enumerate}
\item The expression for $\lambda$ equals the expected number of occurrences of $\tau$ in a uniform random permutation of $[n]$, and thus is the same for all patterns of length $m$.
\item The expression for $d_1$ cannot be improved using our approach. 
\item We are not aware of any efficient means to calculate the values $L_s(\tau)$ in general.  For a simple and explicit upper bound, applicable for all patterns $\tau$ of length $m$, we have used 
\begin{equation}\label{d2:crude}
 d_2 \leq \binom{n}{m} \sum_{s=1}^{m-1} \binom{n-m}{m-s} \binom{m}{s}\frac{s!}{m!^2}, 
 \end{equation}
but we suspect that this bound can be improved. 
\end{enumerate}
\end{rmk}

It is tempting to conjecture that Theorem~\ref{theorem:asymptotic-general} holds even when $\lambda$ tends to some fixed positive constant, but we suspect this is not possible, as we now demonstrate.

\begin{lemma}[\cite{McKayPersonal}]
\label{lemma:mckay}
Fix any $t>0$ and let $\lambda=\binom{n}{m}/m!$.
Then $\lambda\rightarrow t$ as $n\to\infty$ provided 
\begin{equation}\label{j:asymptotic}
m \sim e \sqrt{n} - \frac{1}{4}\log(n) -\frac{1}{2}\log(2\pi t) - \frac{1}{4} e^2 - \frac{1}{2}\quad\text{as }n\rightarrow\infty.
\end{equation}
\end{lemma}

\begin{lemma}\label{lemma:d1bound}
Suppose $n$, $m$, and $n-m$ tend to infinity and $d_1$ is as defined in Theorem~\ref{theorem:asymptotic-general}. 
Then we have 
\[ d_1 \sim \lambda^2\left(1 - e^{-m^2/n}\right). \]
In particular, for any $\epsilon>0$ and $m \geq (e+\epsilon) \sqrt{n}$, we have $d_1 \to 0,$ and for 
$m\sim e\sqrt{n}-\frac{1}{4}\log(n)$, we have $d_1 \to c \in (0,\infty)$.  
\end{lemma}

It follows that a necessary condition for $d_1$ to tend to zero is that
\begin{equation}
\label{possible:j}m \gtrsim e \sqrt{n} - \left(\frac{1}{4}-\epsilon\right)\log(n)\ \quad\text{for any }\epsilon>0,
\end{equation}
where the notation $a \gtrsim b$ means that both $a \geq b$ and $a \sim b$. 
It is also well known, see \cite{VershikKerov,LoganShepp}, that the typical size of the longest increasing subsequence of a random permutation of size~$n$ is asymptotically of order $2 \sqrt{n},$ and so one cannot have a Poisson limit theorem which applies to the increasing pattern $12\ldots m$ with $m < 2 \sqrt{n}$.  
It is interesting to investigate the behavior in the gap, i.e., for $m \sim c \sqrt{n}$ with any $2 < c < e$, and we leave this as an open problem.

We also have an analogous theorem for consecutive patterns. 

\begin{thm}\label{theorem:asymptotic-consecutive}
Assume $n \geq m\geq 3$ and $\tau$ is any pattern of length $m$.  
Define 
\begin{align*}
\overline{\lambda} & = \frac{n-m}{m!},  \qquad 
\overline{d}_1  = \frac{2mn - 3m^2 +m}{m!^2}, \qquad \overline{d}_2  =  \sum_{s=1}^{m-1} \frac{(n-2m+s)\, 2\, \overline{L}_s(\tau)}{(2m-s)!},
\end{align*}
and 
\[ \overline{D}_{n,m} = \min(1, 1/\overline{\lambda})(\overline{d}_1+\overline{d}_2). \]
Then we have 
\begin{align}
\label{bounds-consec}
n!\, e^{-\overline{\lambda}}\left(1 - e^{\overline{\lambda}}\, \overline{D}_{n,m}\right) \leq |\symmetricntilde(\tau)| &\leq n!\, e^{-\overline{\lambda}}\left(1 + e^{\overline{\lambda}}\, \overline{D}_{n,m}\right).
\end{align}

In addition, fix any $t>0$, and define $M(t,n) := \left\lfloor \frac{\log(n/t)}{\log\log(n/t) - \log\log\log(n/t)} - \frac{1}{2}\right\rfloor$.  
Let $m \equiv m(n)$ be some increasing, integer-valued function of $n$ such that $m \geq M(t,n)$. 
For any sequence of patterns $\tau_n\in\symmetricm$, we have $\overline{D}_{n,m} \to 0$, and thus 
\[ \left| \frac{|\symmetricntilde(\tau)|}{n!} - e^{-\overline{\lambda}} \right| = o(1)\quad\text{as }n\to\infty. \]
Furthermore, define $T_n(\tau,k)$ as the number of permutations in $\symmetricn$ where pattern $\tau$ occurs exactly $k$ times, and take $m(n) = M(t,n)$.  
For any sequence of patterns $\tau_n\in\symmetricm$, we have 
\[ \sum_{k=0}^\infty \left| \frac{T_n(\tau, k)}{n!} - \frac{t^k}{k!} e^{t} \right| = o(1)\quad\text{as }n\to\infty. \]
\end{thm}

In Section~\ref{subsection_Mallows}, we present an analogous result for permutations chosen according to the Mallows($q$) distribution. 

\subsection{Main results on uniform permutations} 
\label{poisson limits}

Theorems~\ref{theorem:asymptotic-general} and \ref{theorem:asymptotic-consecutive} provide an asymptotic analysis for sequences of patterns which also grow in size. 
It is too much to expect a general asymptotic formula for any fixed pattern---we have already noted the difficulty of nailing down the asymptotic growth of 1324-avoiding sets---but Poisson approximation, see Section \ref{section:poisson}, provides a general approach for obtaining preasymptotic bounds on various quantities when all sizes are fixed.

\begin{thm}\label{theorem:main} 
Assume $n \geq m \geq 3$ and $\tau$ is any pattern of length $m$.   Recalling the definition of $X_{\alpha}$ in \eqref{eq:indicator}, we let $\mathbb{X} = (X_\alpha)_{\alpha\in J_m}$ be the collection of all such variables and 
 $\mathbb{B} = (B_\alpha)_{\alpha\in  J_m}$ be an independent Bernoulli process with marginal distributions satisfying $\e B_\alpha = \e X_\alpha$ for all $\alpha \in  J_m$.  
We have 
\begin{align}
\label{bounds-TV}
d_{TV}(\mathcal{L}(\mathbb{X}), \mathcal{L}(\mathbb{B})) &\leq 4D_{n,m} + \frac{2\lambda}{m!},
\end{align}
where $D_{n,m}$ and $\lambda$ are as defined in Theorem~\ref{theorem:asymptotic-general}. 
Furthermore, with $W$ defined as in~\eqref{eq:W}, and for $Y$ a Poisson random variable with mean $\lambda = \e W$, we have
\[ d_{TV}(\mathcal{L}(W), \mathcal{L}(Y)) \leq 2D_{n,m}.  \]
\end{thm}

\begin{thm}\label{consecutive patterns}
Assume $n \geq m\geq 3$ and $\tau$ is any pattern of length $m$.  Recall the definitions of $\overline{\mathbb{X}}$, $\overline{\mathbb{B}}$, and $\overline{W}$ given in Section~\ref{section:definitions}. 
We have 
\begin{align}
\label{uniform consecutive bounds}
d_{TV}\left(\LOX), \LOB\right) & \leq 4\overline{D}_{n,m} + \frac{2\overline{\lambda}}{m!},
\end{align}
where $\overline{D}_{n,m}$ and $\overline{\lambda}$ are as defined in Theorem~\ref{theorem:asymptotic-consecutive}.
Let $\overline{Y}$ denote a Poisson random variable with parameter $\overline{\lambda}$.  We have
\[ d_{TV}\left(\mathcal{L}\left(\overline{W}\right), \mathcal{L}\left(\overline{Y}\right)\right) \leq 2\overline{D}_{n,m}.  \]
\end{thm}

\subsection{Main results on Mallows permutations} 
\label{subsection_Mallows}
In Section \ref{section:Mallows}, we discuss several special properties of the Mallows distribution that are useful for studying consecutive pattern avoidance.
Using these properties, we obtain analogous bounds to those in Theorems \ref{theorem:main} and \ref{consecutive patterns}.

Recall the definition of the {\em restriction} $\Sigma_{n|A}$ of $\Sigma_n$ to a subset $A\subseteq[n]$, as defined at the beginning of Section \ref{section:pattern}.
Also recall $\overline{J}_m$ denotes 
the set of subsets of size $m$ whose elements are consecutive in $\{1,2,\ldots, n\}$.  

\begin{thm}\label{main:mallows:theorem}
Fix $q>0$ and let $\Sigma_n\sim\text{Mallows}(q)$ as defined in \eqref{eq:Mallows}.
For any $m \geq 2$, 
fix any patterm $\tau_m$ of size $m$.  For any $\alpha \in \overline{J}_m$, let
\[ X_\alpha = \mathbb{I}(\red(\Sigma_{n|\alpha})=\tau_m)\]
and $W = \sum_{\alpha \in \overline{J}_m} X_\alpha$, and let $Y$ be an independent Poisson random variable with expected value $\lambda = \e W$.  
Then
\[ d_{TV}(\mathcal{L}(W), \mathcal{L}(Y)) \leq 2(b_1+b_2), \]
where
\begin{align}
\notag \lambda&=(n-m)\frac{q^{|\inv(\tau_m)|}}{I_m(q)}, \quad
 b_1 = \frac{(2 m n-3 m^2+m)\, q^{2 |\inv(\tau_m)|}}{I_m(q)^2}, \quad 
 b_2 = \sum_{s=1}^{m-1} (n-2m+s) \sum_{\rho \in \overline{L}_s(\tau_m)} \frac{q^{|\inv(\rho)|}}{I_{2m-s}(q)},
\notag  
 \end{align}
 where $I_m(q)$ is as defined following \eqref{eq:Mallows}.
\end{thm}

Asymptotic formulas and Poisson limit theorems for general Mallows permutations depend on the interplay  between the parameters $n$, $m$, $|\inv(\tau)|$, and $q$.
In particular, we need the expected number of occurrences to converge to a constant $\lambda\in(0,\infty)$.
In the case of consecutive pattern avoidance, the expected number of occurrences of a pattern $\tau\in\symmetricn$ in $\Sigma_n\sim\text{Mallows}(q)$ is
\[\lambda=(n-m)q^{|\inv(\tau)|}/I_m(q),\]
which, for $m$ fixed, produces non-trivial limiting behavior whenever 
\begin{align*}
q\sim n^{-1/|\inv(\tau)|}\quad&\text{or}\quad q\sim n^{1 / \left({m\choose2}-|\inv(\tau)|\right)}\quad\text{as }n\to\infty.
\end{align*}
We can also allow $m$ to vary and keep $q$ fixed so that
\begin{align*}
|\inv(\tau)|\sim-\log(n)/\log(q)\quad&\text{or}\quad|\inv(\tau)|\sim-\log(n)/\log(q)+m^2/2\quad\text{as }n\to\infty. 
\end{align*}
Combined with Theorem \ref{main:mallows:theorem}, these observations yield Theorem~\ref{cor:mallows} below.

\begin{thm}\label{cor:mallows}
Let $m\equiv m(n)$ be a non-decreasing integer-valued sequence, $\tau_n\in\symmetricm$ be a sequence of patterns, and $q\equiv q(n)$ be a sequence of parameters.
For each $n\geq1$, let $\Sigma_n$ be a random permutation from the Mallows distribution \eqref{eq:Mallows} with parameter $q(n),$ with $\overline{\mathbb{X}}_q$ and $\overline{\mathbb{B}}_q$ defined analogously.  
For any measurable function $h : \{0,1\}^{n-m+1} \to \R$ and Borel set $A \subseteq \R$, we have
\[ \mathbb{P}( h(\overline{\mathbb{X}}_q) \in A) = \mathbb{P}( h(\overline{\mathbb{B}}_q) \in A) + o(1), \]
provided either
\begin{align*}
q(n)&\leq n^{-1/\inv(\tau_{m(n)})}\quad\text{for almost all }n\geq1,\\
q(n)&\geq n^{1/(\binom{m(n)}{2}-|\inv(\tau_{m(n)})|)}\quad\text{for almost all }n\geq1,\\
|\inv(\tau_{m(n)})|&\leq-\log(n)/\log(q(n))\quad\text{for almost all }n\geq1\text{ and }q<1\quad\text{or}\\
|\inv(\tau_{m(n)})|&\geq-\log(n)/\log(q(n))+m(n)^2/2\quad\text{for almost all }n\geq1\text{ and }q>1.
\end{align*}
\end{thm}

In Section~\ref{section:3}, we demonstrate Theorem~\ref{main:mallows:theorem} for all patterns of length 3.  
In Section~\ref{section:detailed}, we compute the bounds in Theorem~\ref{main:mallows:theorem} for the specific patterns 2341 and 23451 and we plot the estimated pattern avoidance probabilities in the appropriate asymptotic regime for $q$ from Theorem~\ref{cor:mallows}.

\subsection{Classical pattern avoidance for Mallows permutations}\label{section:classical-Mallows}

We conclude this section by commenting that our approach is silent about classical pattern avoidance in Mallows permutations.  The reason for this limitation is readily seen by noting the interplay between the Arratia--Goldstein--Gordon Theorem (Theorem \ref{Arratia}) and the homogeneity properties of the Mallows distribution, or lack thereof, highlighted in Section \ref{section:Mallows-prop} below.  Briefly, for any pattern $\tau\in\mathcal{S}_m$ and any $1\leq i_1<\ldots<i_m\leq n$, the probability that $\tau$ occurs in positions $i_1,\ldots,i_m$ of a uniform random permutation of $[n]$ is $1/m!$, regardless of the choice of indices $i_1,\ldots,i_m$.  This property is crucial to computing the constants $b_1$ and $b_2$ in our application of Theorem \ref{Arratia}, and allows us to obtain quantitative bounds when permutations are assumed to be drawn uniformly at random.

The same property does not hold for Mallows permutations.  Instead, Mallows permutations only satisfy the weaker property of {\em consecutive homogeneity}, by which the probability that a pattern $\tau$ occurs in consecutive locations $j,j+1,\ldots,j+m-1$ is the same regardless of $j$; see Section \ref{section:Mallows-prop}.  We know of no general formula for computing the probability that a given pattern appears in non-consecutive locations of a Mallows permutation.  Without such a formula, we have no systematic way to extend our results in this direction, and so we leave this as an open problem.

\section{Poisson approximation}\label{section:poisson}

\subsection{Chen--Stein method}
Stein's method is an approach to proving the central limit theorem that was adapted by Chen to Poisson convergence \cite{chen1975poisson}.  The advantage of this method is that it provides guaranteed error bounds on the total variation distance between the distribution of a sum of possibly dependent random variables and the distribution of an independent Poisson random variable with the same mean.  

\begin{thm}[Chen \cite{chen1975poisson}]\label{ChenStein}
Suppose $X_1, X_2, \ldots, X_n$ are indicator random variables with expectations $p_1, p_2, \ldots, p_n$, respectively, and let $W = \sum_{i=1}^n X_i$.  Let $Y$ denote an independent Poisson random variable with expectation $\lambda = \sum_{i=1}^n p_i$.  Suppose, for each $i \geq 1$, a random variable $V_i$ can be constructed on the same probability space as $W$ such that
\[ \mathcal{L}(1+V_i) = \mathcal{L}(W \mid X_i=1).
\]
Then
\begin{equation}\label{dtv bound} 
d_{TV}(\mathcal{L}(W), \mathcal{L}(Y)) \leq \frac{1-e^{-\lambda}}{\lambda}\sum_{i=1}^n p_i \e |W - V_i|.
\end{equation}
\end{thm}
\ignore{\marginpar{
For the coupling below: suppose $\sigma=21543$ and $i=1$.  So $\sigma(i)=2$ and the coupling says ``swap elements $i$ and $j$'' to get $\sigma'=12543$.  Now the number of fixed points has gone from 1 in $\sigma$ to 3 in $\sigma'$ for a difference of 2.
It seems to me that the displayed equation below should have a difference of 2 whenever $i$ is in a 2-cycle.
}}
\subsection{Fixed points example}\label{section:fixed}
To see how Theorem \ref{dtv bound} can be applied, let $e(n)$ denote the number of fixed-point free permutations of $[n]$.  With $\Sigma_n$ a uniform permutation of $[n]$, we define indicator random variables  
\[X_i = \mathbb{I}(\text{ $i$ is a fixed point of $\Sigma_n$ }), \qquad i=1,\ldots,n.
\]
(Note that these random variables are \emph{not} independent.)  
We then define the sum
\[ W = \sum_{i=1}^n X_i \]
so that $P(W = 0) = e(n)/n!$ and $\lambda = \mathbb{E} W=\sum_{i=1}^n \frac{1}{n} = 1$, the expected number of fixed points.  
Even before we proceed with the bound, we obtain the heuristic estimate of $n!\, e^{-1}$ for $e(n)$, just as in \eqref{eq:derangement}.

To apply Theorem~\ref{ChenStein}, we need to construct an explicit coupling of $W$ and $1+V_i$ on the same probability space.  This is done for more general restrictions in \cite{barbour1992poisson}, but we shall write out the full calculation on fixed points to illustrate the basic premise.

We take $1+V_i$ to be the random sum $W$ conditioned on $X_i = 1$.  For a random permutation $\sigma$, suppose $\sigma(i) = j$, for some $j\in [n]$.  The coupling is: swap elements $i$ and $j$.  The resulting permutation has the same marginal distribution as a random permutation conditioned to have a fixed point at $i$.  In fact, $|W - V_i| \in \{0,1,2\}$ for each $i$ since we modify at most 2 elements, and the elements not involved in the swap cancel out (i.e., any fixed points occurring on indices other than these swapping positions remain unchanged).  Let us denote the random variables after the coupling by $X_1', X_2', \ldots, X_n'$; that is, $\mathcal{L}(X_j') = \mathcal{L}(X_j \mid X_i=1)$, so that $1+V_i = \sum_{j=1}^n X_j'$.  
We have 
\[
|W - V_i| = \left|X_i+\sum_{k\neq i} (X_k - X_k')\right| = |X_i + X_J\mathbb{I}(J = \sigma(i), J\neq i)| = \begin{cases} 0, & \sigma(i) \neq i, \text{$i$ not in a 2-cycle}, \\
				      1, & \sigma(i) = i, \\
				       2, & \text{$i$ in a 2-cycle}.
				       \end{cases}
				       \]
The probability that two given elements $i$ and $j$ are part of a 2-cycle is precisely $1/(n(n-1))$, and the probability that $i$ is part of a 1-cycle is $1/n$.  Thus, 
\[ \e |W - V_i| = \frac{1}{n}+\frac{2}{n} = \frac{3}{n}, \]
and equation~\eqref{dtv bound} becomes
\[ 
d_{TV}(\mathcal{L}(W), \mathcal{L}(Y)) \leq (1 - e^{-1}) \sum_{i=1}^n \frac{1}{n} \frac{3}{n} = \frac{3(1-e^{-1})}{n}. 
\]
For all $n \geq 1$, we have 
\[
|P(W=0) - P(Y=0)| = \left| \frac{e(n)}{n!} - e^{-1} \right| \leq \sup_i | P(W=i) - P(Y=i)| \leq d_{TV}(W,Y) \leq \frac{3(1-e^{-1})}{n}.\]
Rearranging yields
\[n! e^{-1} - 3(n-1)! (1-e^{-1}) \leq e(n) \leq n! e^{-1} + 3(n-1)! (1-e^{-1}). \]
Note that this is a guaranteed error bound that holds for all $n\geq 1$, and as a corollary we get $e(n) = n! e^{-1}(1+o(n^{-1}))$.

The error bounds derived from the Chen--Stein method can be improved in special cases, e.g., $e(n)$ above can be obtained exactly by rounding $n!/e$ to the nearest integer for all $n \geq 1$, but the appeal of Poisson approximation is that it applies more generally.

\subsection{The Arratia--Goldstein--Gordon Theorem}
Arratia, Goldstein, \& Gordon \cite{ArratiaGoldstein} provide another approach to Poisson approximation that is sometimes more practical for Poisson approximation.

\begin{thm}[Arratia, Goldstein, \& Gordon \cite{ArratiaGoldstein}]\label{Arratia}
Let $I$ be a countable set of indices and, for each $\alpha \in I$, let $X_\alpha$ be an indicator random variable.
Let $\mathbb{X} = (X_\alpha)_{\alpha\in I}$ denote a collection of (possibly dependent) Bernoulli random variables and 
let $\mathbb{B} = (B_\alpha)_{\alpha\in I}$ denote a collection of independent Bernoulli random variables with marginal distributions which satisfy $\e B_\alpha = \e X_\alpha$ for all $\alpha \in I$.  
Define $p_\alpha := \e X_\alpha = P(X_\alpha = 1) >0$ and $p_{\alpha\beta} := \e X_\alpha X_\beta$.  
Also define $W := \sum_{\alpha \in I} X_\alpha$ and $\lambda := \e W = \sum_{\alpha\in I} p_\alpha$.  For each $\alpha \in I$, define sets $D_\alpha \subset I$ and the quantities
\begin{align}
\notag 		 b_1 &:= \sum_{\alpha \in I} \sum_{\beta \in D_\alpha} p_\alpha p_\beta, \\
\label{eq:b2} 	 b_2 &:= \sum_{\alpha \in I} \sum_{\alpha \neq \beta\in D_\alpha} p_{\alpha \beta}, \quad\text{and}\\
\notag		 b_3 &:= \sum_{\alpha\in I} \e |\e\{X_\alpha - p_\alpha \mid \sigma(X_\beta: \beta \notin D_\alpha) \} |,
\end{align}
where $\sigma(X_\beta: \beta \notin D_\alpha)$ denotes the smallest  $\sigma$-algebra containing $X_\beta$. 
We have
\[ d_{TV}(\mathcal{L}(\mathbb{X}), \mathcal{L}(\mathbb{B})) \leq 2(2b_1 + 2b_2 + b_3) + 2 \sum_{\alpha \in I} p_\alpha^2. \]
Furthermore, let $Y$ denote an independent Poisson random variable with mean $\lambda = \e W$.  
We have 
\[ d_{TV}(\mathcal{L}(W), \mathcal{L}(Y)) \leq 2(b_1+b_2+b_3), \]
and also 
\[ |P(W = 0) - P(Y=0)| \leq \left(b_1+b_2+b_3\right) \frac{1-e^{-\lambda}}{\lambda}. \]
\end{thm}

In our applications, we are able to define sets $D_\alpha$, $\alpha\in I,$ so that $b_3 = 0$ always holds;
whence, the calculations of the bounds in our theorems require only calculations involving first and second (unconditioned) moments.  
For uniform random permutations this is straightforward, but establishing the analogous properties for consecutive patterns under random Mallows permutations is less obvious. 

\section{Consecutive pattern avoidance of Mallows permutations}
\label{section:Mallows}

For any permutation $\sigma=\sigma_1\cdots\sigma_n$, we define its {\em reversal} by $\sigma^r=\sigma_n\cdots\sigma_1$.
By the definition of the Mallows distribution in \eqref{eq:Mallows}, it is apparent that $P_n^q(\sigma)=P_n^{1/q}(\sigma^r)$ for all $\sigma\in\symmetricn$, and so we can focus on the case $0<q\leq 1$ in our analysis.

\subsection{Sequential construction}
The Mallows distribution \eqref{eq:Mallows} enjoys several nice properties that are amenable to the study of pattern avoidance.
These properties are readily observed by the following sequential constructions, both of which are well known and have been leveraged in previous studies of the Mallows distribution; see, for example, \cite{Bhatnagar2014,GnedinOlshanski2009}.
While the properties below are well known, we are not aware of their appearance in relation to pattern avoidance.
We provide proofs for completeness.

For $q>0$, we say that random variable $X$ has the {\em truncated Geometric($q$)} distribution on $[n],$ written as $X~\sim~\text{Geometric}(n,q)$, when the point probabilities of $X$ are given by
\begin{equation}\label{eq:finite-geometric}
P^{n,q}(X=k)=q^{k-1}/(1+\cdots+q^{n-1}),\quad k=1,\ldots,n.
\end{equation}
A Mallows permutation can be generated from the truncated Geometric distribution in two ways, which we call the {\em ordering} and {\em bumping} constructions.

For the ordering construction, we generate $X_1,X_2,\ldots$ independently, with each $X_n$ distributed as $\text{Geometric}(n,1/q)$.
To initialize, we have $\Sigma_1=1$, the only permutation of $[1]$.
Given $\Sigma_n=\sigma_1\cdots\sigma_n$ and $X_{n+1}=k$, we define 
\[\Sigma_{n+1}=\sigma_1\cdots\sigma_{k-1}(n+1)\sigma_k\cdots\sigma_n.\]
For every $n=1,2,\ldots$, it  is apparent that $\Sigma_n$ is a Mallows$(q)$ permutation because the probability that element $n+1$ appears in position $k$ of $\Sigma_{n+1}$ is
\[P(\Sigma_{n+1}(k)=n+1)=P^{n+1,1/q}(X=k)=P^{n+1,q}(X=n+1-k)=q^{n+1-k}/(1+q+\cdots+q^n).\]
Since $X_1,\ldots,X_n$ are chosen independently and each event $\{\Sigma_n=\sigma\}$ corresponds to exactly one sequence $X_1,\ldots,X_n$, we observe
\[P(\Sigma_n=\sigma)=q^{|\inv(\sigma)|}/I_n(q),\quad \sigma\in\symmetricn,\]
as in \eqref{eq:Mallows}.

\begin{defn}[Mallows process]\label{defn:Mallows process}
A collection $(\Sigma_n)_{n\geq1}$ generated by the ordering construction for fixed $q>0$ is called a {\em Mallows$(q)$ process}.
\end{defn}

For the bumping construction, we generate $X_1,X_2,\ldots$ independently with each $X_n$ distributed as $\text{Geometric}(n,1/q)$ as before, and again we initialize with $\Sigma_1=1$.
Given $\Sigma_n=\sigma_1\cdots\sigma_n$ and $X_{n+1}=k$, we obtain $\Sigma_{n+1}$ by appending $k$ to the end of $\Sigma_n$ and ``bumping'' all elements of $\Sigma_n$ that are greater or equal to $k$.
More formally, the permutation $\Sigma_n$ and the variable $X_{n+1}$ give rise to the updated permutation  $\Sigma'_1\cdots\Sigma'_nX_{n+1}$, where
\[\Sigma'_j=\left\{\begin{array}{cc} \Sigma_j+1,& \Sigma_j\geq X_{n+1},\\
\Sigma_j,& \text{otherwise.}
\end{array}\right.\]
For example, if $\Sigma_5=24135$ and $X_6=3$, then $\Sigma_{6}=251463$.
Again, the resulting distribution of $\Sigma_n$ is Mallows($q$) because $X_{n+1}=k$ introduces exactly $n+1-k$ new inversions in $\Sigma_{n+1}$ and $X_1,X_2,\ldots$ are generated independently.

\subsection{Properties of Mallows permutations}\label{section:Mallows-prop}

Throughout this section, we let $(\Sigma_n)_{n\geq1}$ be a family of random permutations so that each $\Sigma_n$ is a permutation of $[n]$.
We say that $(\Sigma_n)_{n\geq1}$ is {\em consistent} if for all $1\leq m\leq n$
\begin{equation}\label{eq:consistent}
P(\Sigma_{n|[m]}=\sigma)=P(\Sigma_m=\sigma),\quad\sigma\in\symmetricm.
\end{equation}
It is immediate from the ordering construction that the Mallows process $(\Sigma_n)_{n\geq1}$ is consistent for every $q>0$.

Recall the reduction map described in Section \ref{section:intro}.
We call $(\Sigma_n)_{n\geq1}$ {\em homogeneous} if for all $1\leq m\leq n$ and every subsequence $1\leq i_1<\cdots<i_m\leq n$
\begin{equation}\label{eq:homo}
P(\red(\Sigma_{n}(i_1)\cdots\Sigma_{n}(i_m))=\sigma)=P(\Sigma_m=\sigma),\quad\sigma\in\symmetricm.
\end{equation}
We call $(\Sigma_n)_{n\geq1}$ {\em consecutively homogeneous} if \eqref{eq:homo} holds only for consecutive subsequences $i_1, i_1+1,\ldots,i_1+m-1$.

\begin{lemma}\label{lemma:homo}
The Mallows($q$) process is consecutively homogeneous for all $q>0$ and homogeneous for $q=1$.
\end{lemma}

\begin{proof}
The $q=1$ case corresponds to the uniform distribution, which is well known to be homogeneous.
For arbitrary $q>0$, consider the event $\{\red(\Sigma_{n}(j)\cdots\Sigma_{n}(j+m-1))=\sigma\}$ for some $\sigma\in\symmetricn$.
By the ordering construction, we can first generate $\Sigma_m=\Sigma_{m}(1)\cdots\Sigma_{m}(m)$ from the Mallows($1/q$) distribution on $[m]$.
We then obtain $\Sigma_{m+j-1}$ from $\Sigma_m$ using the bumping construction for Mallows($1/q$) distribution.
Thus, we have $\Sigma_{m+j-1}\sim\text{Mallows}(1/q)$ and its reversal $\Sigma_{m+j-1}^r\sim\text{Mallows}(q)$ with $\red(\Sigma^r_{m+j-1}(j)\cdots\Sigma^r_{m+j-1}(m+j-1))=\Sigma_m(m)\cdots\Sigma_m(1)\sim\text{Mallows}(q)$.
Finally, we obtain $\Sigma_n$ by augmenting $\Sigma_{m+j-1}^r$ according to the bumping construction, so that 
\[\red(\Sigma_{n}(j)\cdots\Sigma_{n}(m+j-1))=\red(\Sigma_{m+j-1}^r(j)\cdots\Sigma^r_{m+j-1}(m+j-1))=\Sigma_m(m)\cdots\Sigma_m(1)\sim\text{Mallows}(q).\]
This completes the proof.
\end{proof}

We say that $\Sigma_n$ is {\em dissociated} if $\Sigma_{n|A}$ and $\Sigma_{n|B}$ are independent for all non-overlapping subsets $A,B\subseteq[n]$.
If, instead, $\Sigma_{n|A}$ and $\Sigma_{n|B}$ are independent only when $A$ and $B$ are disjoint and each consists of consecutive indices, then we call $\Sigma_n$ {\em weakly dissociated}.

\begin{lemma}\label{lemma:dissociated}
For all $n\geq1$, the Mallows($q$) distribution on $\symmetricn$ is weakly dissociated for all $q>0$ and dissociated for $q=1$.
\end{lemma}

\begin{proof}
For $i'>i\geq1$ and $m,m'\geq0$ satisfying $i+m-1<i'$ and $i'+m'-1\leq n$, let $A=\{i,i+1,\ldots,i+m-1\}$ and $B=\{i',i'+1,\ldots,i'+m'-1\}$.
For any $n\geq1$, we can construct a Mallows($q$) permutation of $[n]$ by first generating $\Sigma_{i+m-1}$, for which we know that $\red(\Sigma_{i+m-1}(i)\cdots\Sigma_{i+m-1}(i+m-1))\sim\text{Mallows}(q)$ by Lemma \ref{lemma:homo}.
We then construct $\Sigma_n$ from $\Sigma_{i+m-1}$ by the bumping construction.
Since bumping does not affect the reduction of any part of $\Sigma_{n}(1)\cdots\Sigma_{n}(i+m-1)$, we have
\begin{eqnarray*}
\lefteqn{P(\red(\Sigma_{n}(i)\cdots\Sigma_{n}(i+m-1))=\sigma\mid\red(\Sigma_{n}(i')\cdots\Sigma_{n}(i'+m'-1))=\sigma'\}=)}\\
&=&
P(\red(\Sigma_{i+m-1}(i)\cdots\Sigma_{i+m-1}(i+m-1))=\sigma\mid\red(\Sigma_{n}(i')\cdots\Sigma_{n}(i'+m'-1))=\sigma')\\
&=&P(\Sigma_{m}(1)\cdots\Sigma_{m}(m)=\sigma),
\end{eqnarray*}
proving that $\Sigma_n$ is weakly dissociated.
Dissociation of the uniform distribution ($q=1$) is well known and so we omit its proof.
The proof is complete.
\end{proof}

Together, the above properties facilitate study of consecutive pattern avoidance for Mallows permutations with arbitrary $q>0$.
For example, the pattern $231$ has probability $q^2/(1+2q+2q^2+q^3)$ to occur in any stretch of three consecutive positions of a Mallows($q$) permutation.
Since there are $n-2$ consecutive patterns of length 3 in a permutation of $[n]$, the expected number of occurrences is $(n-2)q^2/(1+2q+2q^2+q^3)$.
For large $n$ and small $q$, this expected value behaves asymptotically as $nq^2$, so that taking $q\sim1/\sqrt{n}$ gives an expected number on the order of a constant.
 When $q$ is large, the expected number of occurrences behaves as $n q^{-1}$ for large $n$, and taking $q\sim n$ gives an expected number on the order of a constant.

\subsection{Poisson convergence theorems}

Theorems~\ref{main:mallows:theorem} and~\ref{cor:mallows} follow by combining the above properties of Mallows permutations with Theorem~\ref{Arratia}.
The calculations and resulting bounds for the general Mallows measure follow the same program as the uniform case, with the key distinction that we only consider {\em consecutive} patterns for the general Mallows distribution; see Section \ref{section:classical-Mallows}.  
Unlike the uniform setting, the bounds for the Mallows distribution depend non-trivially on the parameter $q$ and the structure of $\tau$.
It is more fruitful to illustrate this dependence with specific examples than to regurgitate the same proof for Mallows permutations.

\subsubsection{Monotonic patterns under Mallows distribution}\label{section:monotonic}
 
Consider the set of permutations that avoid the pattern 123.  There are no inversions, and the size of the pattern is $3$; thus, the probability of seeing this pattern in any given set of three consecutive indices of a Mallows($q$) permutation is $1/I_3(q)$.  
We also need to consider second moments, i.e., the probability of seeing two 123 patterns.  
By Lemma~\ref{lemma:dissociated} we need only consider overlapping sets of indices.
There are two cases, either two indices overlap or one does. 
If two indices overlap and the first three and last three both reduce to pattern 123, then the segment must reduce to 1234.
Similarly, if one index overlaps, then the segment must reduce to 12345.

The results below extend this argument to monotonic patterns. 

\begin{lemma}
Fix $q>0$ and let $\Sigma_n\sim\text{Mallows}(q)$.
For each $m \geq 1$, let $\tau_m$ denote the pattern $12\cdots m$.  
For each $\alpha \in \overline{J}_m$, define
\[ X_\alpha = \mathbb{I}(\red(\Sigma_{n|\alpha})=\tau_m). \]
(Recall that $\Sigma_{n|\alpha}$ denotes the restriction of $\Sigma_n$ to the subset of indices $\alpha$.)
For a random permutation generated using the Mallows measure, we have
\[ \e X_\alpha = \frac{1}{I_m(q)},  \qquad \alpha \in \overline{J}_m, \]
and for $\alpha, \beta \in \overline{J}_m$, $\alpha \neq \beta$, we have
\[ \e X_\alpha X_\beta = 
\begin{cases} 
	\frac{1}{I_m(q)^2}, & \text{if $\alpha, \beta$ have no overlapping elements} \\
	\frac{1}{I_{2m-s}(q)}, & \text{if $\alpha, \beta$ have exactly $s$ overlapping elements, $s = 1, 2, \ldots, m-1$.}
\end{cases}
\]
\end{lemma}

\begin{proof}
The expression for $\mathbb{E}X_{\alpha}X_{\beta}$ when $\alpha$ and $\beta$ do not overlap is a consequence of the weak dissociation property of Mallows permutations (Lemma \ref{lemma:dissociated}), whereby
\[\e X_{\alpha}X_{\beta}=\e X_{\alpha} \e X_{\beta} = P(\red(\Sigma_{n|\alpha})=\tau_m)^2=(1/I_m(q))^2.\]
When $\alpha$ and $\beta$ overlap in $s$ elements, the event $\{X_{\alpha}=X_{\beta}=1\}$ requires that both $\Sigma_{n|\alpha}$ and $\Sigma_{n|\beta}$ reduce to the increasing permutation, which can occur only if $\Sigma_{n|\alpha\cup\beta}$ reduces to the increasing permutation of $2m-s$.
\end{proof}

\begin{prop}\label{prop:increasing}
Fix $q>0$ and let $\Sigma_n\sim\text{Mallows}(q)$.
For any $m \geq 2$, 
let $\tau_m$ be the increasing pattern $12\cdots m$.  
For any $\alpha \in \overline{J}_m$, let
\[ X_\alpha = \mathbb{I}(\red(\Sigma_{n|\alpha})=\tau_m). \]
Let $W = \sum_{\alpha \in \overline{J}_m} X_\alpha$ and let $Y$ be an independent Poisson random variable with expected value $\lambda = \e W$.  
Then
\[ d_{TV}(\mathcal{L}(W), \mathcal{L}(Y)) \leq 2(b_1+b_2), \]
where
\begin{align}
\notag \lambda&=\frac{n-m}{I_m(q)}, \qquad 
 b_1 = \frac{n_1}{I_m(q)^2}, \qquad 
 b_2  =n_2 \sum_{s=1}^{m-1}\frac{1}{I_{2m-s}(q)},
\end{align}
 and $n_1$ and $n_2$ are given by
 \begin{align}
\notag n_1 &= 2 m n- 3 m^2+m \quad\text{and}\\
\label{mono n2} n_2 &= 3 m - 3 m^2 - 2 n + 2 m n.
 \end{align}
\end{prop}
\begin{proof}
When $2m-1 \leq n$, we have $n_1 = 2\sum_{s=1}^m (n-2m+s) = 2 m n- 3 m^2+m$, and similarly $n_2 = 2\sum_{s=1}^{m-1}(n-2m+s) = 3 m - 3 m^2 - 2 n + 2 m n.$  The factor of $2$ is from exchanging the role of $\alpha, \beta$.
When $2m-1 > n$, the stated expressions for $n_1$ and $n_2$ are still valid upper bounds, but they can be  improved.
\end{proof}

It is straightforward to state the complementary result about the decreasing pattern $m\cdots 21$.

\begin{prop}\label{prop:decreasing}
Fix $q>0$ and let $\Sigma_n\sim\text{Mallows}(q)$.
For any $m \geq 2$, 
let $\eta_m$ be the decreasing pattern $m\cdots 21$.
For any $\alpha \in J_m$, let
\[ X_\alpha = \mathbb{I}(\red(\Sigma_{n|\alpha})=\eta_m) \]
 and $W = \sum_{\alpha \in J_m} X_\alpha$, and allow $Y$ to denote an independent Poisson random variable with expected value $\lambda = \e W$.  
Then
\[ d_{TV}(\mathcal{L}(W), \mathcal{L}(Y)) \leq 2(b_1+b_2), \]
where
\begin{align*}
\lambda&= (n-m)\frac{q^{{m\choose2}}}{I_m(q)}, \qquad 
 b_1 = \frac{n_1\ q^{\binom{m}{2}}}{I_m(q)^2}, \qquad
 b_2 = n_2\ \sum_{s=1}^{m-1} \frac{q^{\binom{2m-s}{2}}}{I_{2m-s}(q)}, 
\end{align*}
 and $n_1$ and $n_2$ are given by
 \begin{align*}
\notag n_1 &= 2 m n- 3 m^2+m \quad\text{and}\\
\label{mono n2} n_2 &= 3 m - 3 m^2 - 2 n + 2 m n.
 \end{align*}
\end{prop}

\subsubsection{Other patterns of length 3}\label{section:3}
We now demonstrate the dependence of the total variation bound on $q$ for the small patterns 132, 213, 231, and 312.
 In this section, we again recall that $\overline{J}_m$ denotes the set of all $m$-tuples with consecutive elements in $\{1,2,\ldots, n\}$, and for a given $\alpha \in \overline{J}_3$, $X_\alpha$  denotes the indicator random variable defined in \eqref{eq:indicator}.

For $\tau=132$, we have $\e X_\alpha = {q}/{I_3(q)}$, and there can be no consecutive occurrences of $\tau$ that overlap with two indices.  The only possible ways to have one overlapping index are the patterns 13254, 15243, and 14253.  In these cases, we have
\[ \lambda  = \frac{(n-3)\, q}{I_3(q)}\quad\text{and} \]
\[ 
\e X_\alpha X_\beta = \frac{1}{I_5(q)}\times \begin{cases}
q^2, & 13254, \\
q^3, & 14253, \\
q^4, & 15243.
\end{cases}
\]
Letting $W = \sum_{\alpha \in \overline{J}_3} \e X_\alpha$  and defining $Y$ as an independent Poisson random variable with expectation $\lambda = \e W = {(n-3)\, q}/{I_3(q)}$, the total variation distance bound is given by
\[ d_{TV}(\mathcal{L}(W), \mathcal{L}(Y)) \leq 2\left( (3n-13)\, \frac{q^2}{I_3(q)^2} + 2(n-5)\, \frac{q^2+q^3+q^4}{I_5(q)}\right). 
\]
The $2(n-5)$ term comes from the two sets of triplets $\{1,2,3\}$ and $\{2,3,4\}$ for which the overlapping pair can only occur to the right of the elements, and similarly from the two sets of triplets $\{n-2, n-1, n\}$ and $\{n-3, n-2, n-1\}$ for which the overlapping pair can only occur to the left of the elements, and finally the $(n-4-3)$ triplets in between for which the overlapping pairs are both to the left and the right; hence $2+2(n-7)+2 = 2(n-5)$. 
Similarly, the $3n-13$ comes from $2\cdot2 + 3(n-7) + 2\cdot 2$.  
For fixed $t>0$, we have $\lambda\to t$ and $d_{TV}(\mathcal{L}(W), \mathcal{L}(Y)) = O(n^{-1})$, provided $q \sim t\, n^{-1}$ or $q \sim t\, n^{1/2}$.  

For $\tau=213$, we similarly have 
\begin{align*}
\lambda & = \frac{(n-3)\, q}{I_3(q)}, \\
\e X_\alpha X_\beta & =
\frac{1}{I_5(q)}\times \begin{cases}
{q^2}, & 21435, \\
{q^3}, & 31425, \\
{q^4}, & 32415,
\end{cases}\quad \text{and} \\
d_{TV}(\mathcal{L}(W), \mathcal{L}(Y)) & \leq 2\left( (3n-13)\, \frac{q^2}{I_3(q)^2} + 2(n-5)\, \frac{q^2+q^3+q^4}{I_5(q)}\right),
\end{align*}
which for $q \sim t\,n^{-1}$ or $q \sim t\,n^{1/2}$ implies $\lambda \to t$ and $d_{TV}(\mathcal{L}(W), \mathcal{L}(Y)) = O(n^{-1}).$  

For $\tau= 231$:
\begin{align*}
\lambda & = \frac{(n-3)\, q^2}{I_3(q)}, \\
\e X_\alpha X_\beta & = \frac{1}{I_5(q)}\times\begin{cases}
{q^6}, & 34251, \\
{q^7}, & 35241, \\
{q^8}, & 45231,
\end{cases} \quad\text{and}\\
d_{TV}(\mathcal{L}(W), \mathcal{L}(Y)) & \leq 2\left( (3n-13)\, \frac{q^4}{I_3(q)^2} + 2(n-5)\, \frac{q^6+q^7+q^8}{I_5(q)}\right), 
\end{align*}
which for $q \sim t^{1/2} n^{-1/2}$ or $q \sim n\, t^{1/2}$ implies $\lambda \to t$ and $d_{TV}(\mathcal{L}(W), \mathcal{L}(Y)) = O(n^{-1}).$  

And finally for $\tau= 312$:
\begin{align*}
\lambda & = \frac{(n-3)\, q^2}{I_3(q)}, \\
\e X_\alpha X_\beta & = \frac{1}{I_5(q)}\times\begin{cases}
{q^6}, & 51423, \\
{q^7}, & 52413, \\
{q^8}, & 53412,
\end{cases}\quad\text{and} \\
d_{TV}(\mathcal{L}(W), \mathcal{L}(Y)) & \leq 2\left( (3n-13)\, \frac{q^4}{I_3(q)^2} + 2(n-5)\, \frac{q^6+q^7+q^8}{I_5(q)}\right), 
\end{align*}
which for $q \sim t^{1/2}\, n^{-1/2}$ or $q \sim n\,t^{1/2}$ implies $\lambda \to t$ and $d_{TV}(\mathcal{L}(W), \mathcal{L}(Y)) = O(n^{-1}).$

\section{Numerical examples}\label{section:cases}

\subsection{Numerical values}

Using Theorem~\ref{theorem:main}, we can estimate $|\symmetricn(\tau)|$ for different sizes of patterns $\tau$.  Table~\ref{table:bounds} shows the lower bound thresholds for several values of $n$ and patterns of size $j$.  
Similarly, using Theorem~\ref{consecutive patterns}, we estimate $|\symmetricntilde(\tau)|$ in Table~\ref{table:consecutive:bounds}.
In the case of $n = 1000$ and $j = 133$, we have more specifically $3.4433\times 10^{2567} \leq |\symmetricn(\tau)| \leq 4.0239\times 10^{2567}$.

\begin{table}[t]
\begin{tabular}{llll}
$n$ & $j$ & lower & $n!$ \\\hline
 100 & 36 & $6.85456\times 10^{157}$ & $9.3326\times 10^{157}$ \\
 1000 & 133 & $3.4433\times 10^{2567}$ & $4.6045\times 10^{2567}$ \\
 10000 & 442 & $8.3847\times 10^{35658}$ & $2.8463\times 10^{35659}$ \\
 100000 & 14353 & $9.9451\times 10^{65657058}$ & $1.2024\times 10^{65657059}$
\end{tabular}
\caption{Bounds on $|\symmetricn(\tau)|$ for $\tau \in \symmetricj$, for different values of $n$ and $j$.}
\label{table:bounds}
\end{table}

\begin{table}[t]
\begin{tabular}{llll}
$n$ & $j$ & lower & $n!$ \\\hline
100      & 6  & $3.98735\times 10^{157}$ & $9.33262\times 10^{157}$ \\
1000    & 7  & $5.77948\times 10^{2566}$ & $4.02387\times 10^{2567}$ \\
10000  & 9  & $2.49966\times 10^{35659}$ & $2.84626\times 10^{35659}$ \\
100000 & 10 & $2.48004\times 10^{456573}$ &$2.82423\times 10^{456573}$ \\
1000000 &11 & $7.34802\times 10^{5565708}$ &$8.26393\times 10^{5565708}$
\end{tabular}
\caption{ Bounds on $|\symmetricntilde(\tau)|$ for $\tau \in \symmetricj$, for different values of $n$ and $j$.}
\label{table:consecutive:bounds}
\end{table}

\subsection{Detailed illustration for the patterns 2341 and 23451}\label{section:detailed}

Propositions \ref{prop:increasing} and \ref{prop:decreasing} give an expression for the total variation bound between the number of occurrences of the increasing and decreasing patterns and an independent Poisson random variable.
In principle, these bounds can be computed exactly for any pattern by way of the Arratia--Goldstein--Gordon theorem (Theorem \ref{Arratia}), and so we need only compute the quantities $b_1$, $b_2$, and $b_3$ as in Theorem \ref{Arratia}.

By Lemma \ref{lemma:dissociated}, all Mallows($q$) permutations are weakly dissociated and, therefore, $b_3\equiv0$ for all patterns in the case of consecutive pattern avoidance.
For any pattern $\tau$, homogeneity of the Mallows measure implies $p_{\alpha}=q^{|\inv(\tau)|}/I_m(q)$ for all $\alpha,$ and so $b_1$ is easy to compute.
The only complication involves the consideration of overlapping patterns in the calculation of $b_2$.
We cannot prove anything more general than Arratia--Goldstein--Gordon for arbitrary patterns; instead, we compute these bounds in the special cases of $\tau=2341$ and $\tau=23451$.
Figure \ref{figure:plots} shows the performance of these bounds at the critical values $q\sim n^{-1/3}$ and $q\sim n^{1/3}$ for $\tau=2341,$ and $q\sim n^{-1/4}$ and $q\sim n^{1/6}$ for $\tau=23451$.

\begin{figure}[t]
	\centering
		\includegraphics[width=0.8\textwidth]{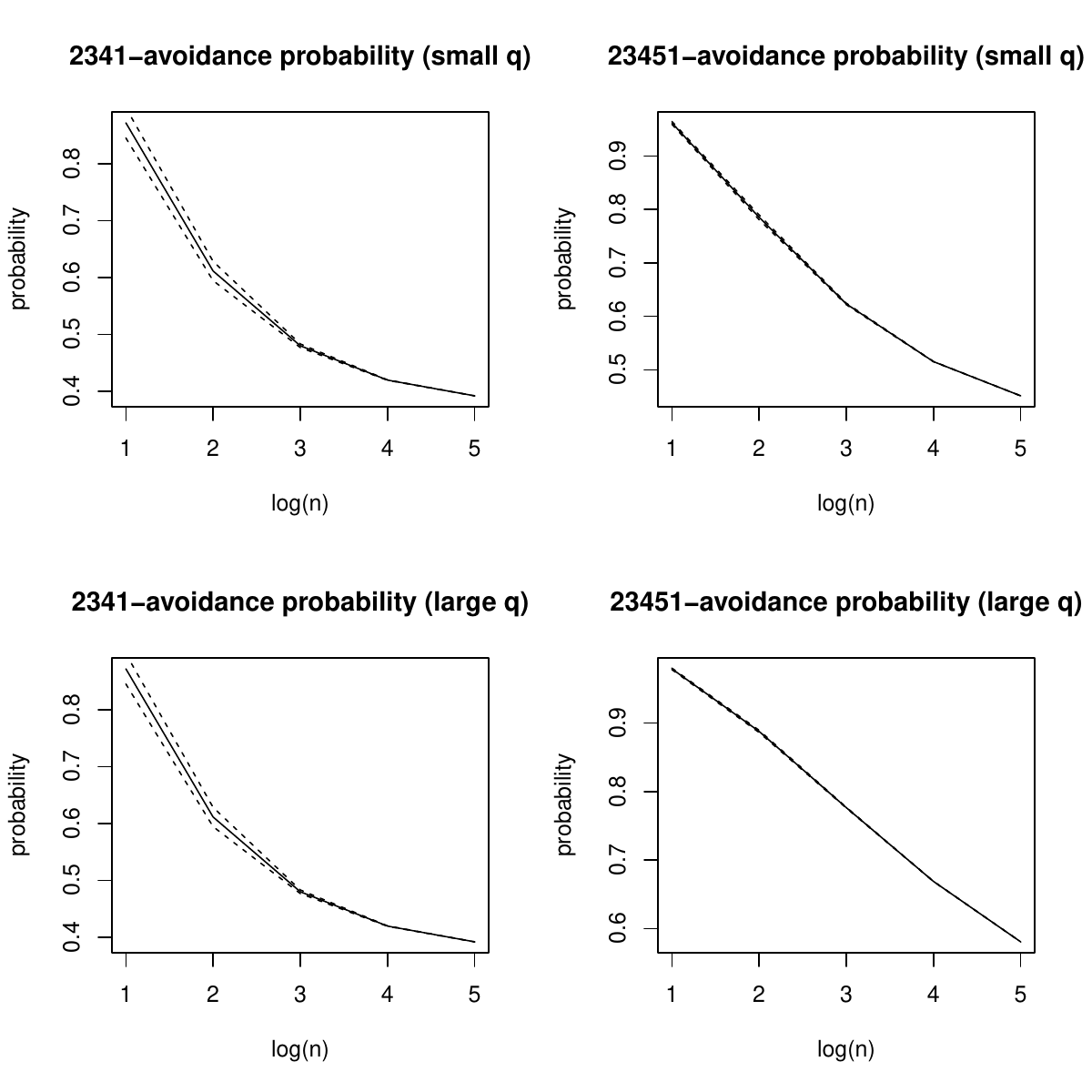}
	\caption{ Plot of pattern avoidance probabilities of Mallows($q$) distribution for: (top left)  pattern $\tau=2341$ with $q=n^{-1/3}$; (top right) pattern $\tau=23451$ with $q=n^{-1/4}$; (bottom left) pattern $\tau=2341$ with $q=n^{1/3}$; and (bottom right) pattern $\tau=23451$ with $q=n^{1/6}$.
The dashed lines represent the upper and lower error bounds from the Arratia--Goldstein--Gordon theorem, and the solid line represent their average, i.e., the heuristic approximation.
In all panels, the horizontal axis is on the logarithmic scale with base 10.  Note that the dashed lines overlap with the solid line in the two figures on the right, indicating that the error bounds are tightly concentrated about the estimated probability.}
	\label{figure:plots}
\end{figure}

\begin{figure}[t]
\includegraphics[width=0.4\textwidth]{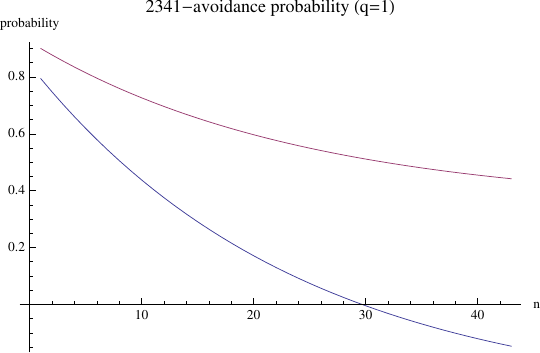}
\includegraphics[width=0.4\textwidth]{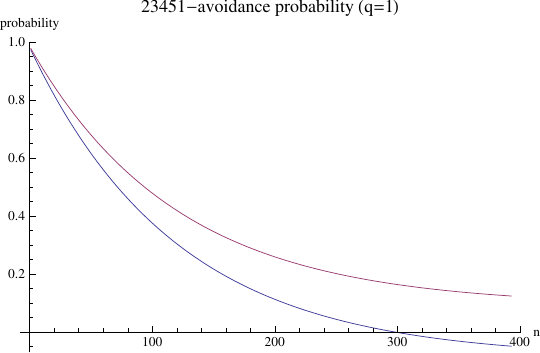}
\caption{ Plot of lower and upper bounds on pattern avoidance probabilities of uniform distribution ($q=1$) for  $\tau=2341$ (left) and $\tau=23451$ (right).}
	\label{figure:plots2}
\end{figure}

\subsubsection{The pattern $2341$}

\begin{table}[t]
\begin{tabular}{c|c||c|c}
permutation & no.\ inversions & permutation & no.\ inversions \\\hline
3452671 & 9 & 3462571 & 10 \\
3472561 & 11 & 3562471 & 11 \\
3572461 & 12 & 4562371 & 12\\
4572361 & 13 & 3672451 & 13\\
4672351 & 14 & 5672341 & 15
\end{tabular}
\caption{List of all permutations that have pattern 2341 in overlapping positions along with the number of inversions.}\label{2341}
\end{table}
For $\tau=2341$, we have $p_{\alpha}=q^{3}/I_4(q)$ and $b_1=(n-4)q^3/I_4(q)$.
The structure of $\tau$ only permits overlap with the first or last position.
Table \ref{2341} lists all permutations that have pattern 2341 in the first 4 and last 4 positions.  These are the only permutations that contribute to $b_2$ in the bound of Theorem \ref{Arratia}.
We assume $n\geq7$.
For positions $5,\ldots,n-5$, each of these overlapping patterns can occur twice; otherwise, the patterns occur only once for a total of $2(n-8)+8 = 2n-8$ possibilities.
There are $6(n-8)+2(5+4+3)=6n-24$ overlapping patterns $\alpha$ and $\beta$ that contribute to $b_1$.
Thus,
\begin{align*}
\lambda&=(n-4)q^3/I_4(q),\\
b_1&=(6n-24)q^{6}/I_4(q)^2,\quad\text{and}\\
b_2&=(2n-8)q^{9}(1+q+2q^{2}+2q^{3}+2q^{4}+q^{5}+q^{6})/I_7(q),
\end{align*}
giving the bounds
\[e^{-\lambda}-(b_1+b_2)\frac{1-e^{-\lambda}}{\lambda}\leq P(W=0)\leq e^{-\lambda} + (b_1+b_2)\frac{1-e^{-\lambda}}{\lambda}.\]

\subsubsection{The pattern $23451$}

\begin{table}[t]
\begin{tabular}{c|c||c|c||c|c}
permutation & no.\ inversions & permutation & no.\ inversions & permutation & no.\ inversions\\\hline
345627891 & 12 & 347825691 & 16 & 456923781 & 18\\
345726891 & 13 & 347925681 & 17 & 467823591 & 19 \\
345826791 & 14 & 348925671 & 18 & 467923581 & 20 \\
345926781 & 15 & 357824691 & 17 &468923571 & 21 \\
346725891 & 14 & 357924681 & 18 & 567823491 & 20\\
346825791 & 15 & 358924671 & 19 & 567923481 & 21\\
346924781 & 16 & 367824591 & 18 & 568923471 & 22 \\
356724891 & 15 & 367924581 & 19 & 378924561 & 21 \\
356824781 & 16 & 368924571 & 20 &  478923561 & 22\\
356924781 & 17 & 457823691 & 18 & 578923461 & 23 \\
456723891 & 16 & 457923681 & 19 & 678923451 & 24\\
456823791 & 17 & 458923671 & 20 &  &\\ 
\end{tabular}
\caption{List of all permutations that have pattern 23451 in overlapping positions along with the number of inversions.}
\label{23451}
\end{table}

For $\tau=23451$, we have $p_{\alpha}=q^{4}/I_5(q)$ and $b_1=(n-5)q^4/I_5(q)$.
The structure of $\tau$ only permits overlap with the first or last position.
Table \ref{23451} lists all permutations that have pattern 23451 in the first 5 and last 5 positions.  These are the only permutations that contribute to $b_2$ in the bound of Theorem \ref{Arratia}.
We assume $n\geq9$.
For positions $5,\ldots,n-5$, each of these overlapping patterns can occur twice; otherwise, the patterns occur only once for a total of $2(n-10)+10 = 2n-10$ possibilities.
There are $8(n-10)+2(7+6+5+4)=8n-36$ overlapping patterns $\alpha$ and $\beta$ that contribute to $b_1$.
Thus,
\begin{align*}
\lambda&=(n-5)q^4/I_5(q),\\
b_1&=(8n-36)q^{8}/I_5(q)^2,\quad\text{and}\\
b_2&=(2n-10)q^{12}(1+q+2q^{2}+3q^{3}+4q^{4}+4q^{5}+5q^{6}+4q^{7}+4q^{8}+3q^{9}+2q^{10}+q^{11}+q^{12})/I_9(q),
\end{align*}
producing the bounds
\[e^{-\lambda}-(b_1+b_2)\frac{1-e^{-\lambda}}{\lambda}\leq P(W=0)\leq e^{-\lambda} + (b_1+b_2)\frac{1-e^{-\lambda}}{\lambda}.\]

\section{Proofs}\label{section:proofs}

\subsection{Bounds on $p_\alpha$ and $p_{\alpha\beta}$}
We first prove several lemmas, from which the theorems follow.
By the homogeneity property of uniform permutations  we have
\[p_\alpha = \e X_\alpha = {1}/{m!}\quad\text{for all }\alpha\in J_m.\] 
To calculate the Poisson rate $\lambda$, we use linearity of expectation:  
there are $\binom{n}{m}$ possible $m$-tuples of elements in $[n]$, and so the expected number of subsets of $m$ elements that reduce to the pattern $k_1k_2\cdots k_m$ is
\[ \lambda = | J_m|\, p_\alpha = \binom{n}{m}/m!.\]

Next, we consider the joint expectation $p_{\alpha\beta} = \e X_\alpha X_\beta.$  
In these calculations, recall the notation for $d_1, d_2$ and $\overline{d}_1, \overline{d}_2$ from Theorems \ref{theorem:asymptotic-general} and \ref{theorem:asymptotic-consecutive}, respectively.   
\begin{lemma}\label{lemma:pab}
Fix $\alpha,\beta\in J_m$ and let $s=1,\ldots, m-1$ denote the number of elements that $\alpha, \beta$ have in common.
 For any such pair $\alpha, \beta$, we have
\begin{equation}\label{pab}
 p_{\alpha\beta} \le \frac{s!}{m!^2}. 
\end{equation}
\end{lemma}
\begin{proof}
First we condition on $X_\alpha$, which contributes a factor of $1/m!$.  
By conditioning on $X_{\alpha}$, we assume that the $s$ common elements are in their proper order with respect to $X_\beta$.  It may so happen that, conditional on $X_\alpha$, no such event can occur, which justifies the inequality. 

Consider first $s = m-1$, i.e., condition on $m-1$ of the entries being in their proper order.  
Assuming that it is possible to realize both events simultaneously, the remaining element has probability $1/m$ of appearing in its correct order. 
For general $s$, conditional on $s$ entries being in their proper order, the probability that the remaining $m-s$ elements appear in their proper order is then $((s+1) (s+2)\cdots m)^{-1}$.  
\end{proof}

\subsection{Proof of Theorem~\ref{theorem:main}}
We have the following lemma.

\begin{lemma}
For $D_{\alpha}$ defined as in Section~\ref{section:definitions} and $b_3$ as in Theorem~\ref{Arratia}, we have $ b_3 = 0$ for all patterns $\tau$.
\end{lemma}

\begin{proof}
This follows from the dissociated property of uniform permutations (Lemma \ref{lemma:dissociated}).
We interpret the conditioning on $\sigma(X_\beta : \beta \notin D_\alpha)$ as the $\sigma$-algebra containing all possible information about just the order of a particular set of three elements.  Since these three elements do not overlap any of the elements in $\alpha$, knowing only their order does not affect $X_\alpha$ because uniform permutations are dissociated.
\end{proof}

\begin{rmk}
Note that the conditioning in the expression for $b_3$ is not the $\sigma$-algebra containing all information about the elements indexed by each tuple.  If it were, then knowing their particular location would have an impact.  However, simply knowing their order does not reveal any more information about $X_\alpha$.
\end{rmk}

\begin{lemma}\label{count:D}
For each $\alpha \in  J_m$, 
\[ |D_\alpha| = \binom{n}{m} - \binom{n-m}{m}. \]
(Note that the value of $D_{\alpha}$ does not depend on $\alpha$.)
\end{lemma}
\begin{proof} 
Fix any $\alpha,\beta\in J_m$ and let $s=1,\ldots,m$ denote the number of elements that $\alpha, \beta$ have in common.
 (This includes the case $\alpha = \beta$.)  
For each $s = 1,2,\ldots,m$, we select any $s$ elements out of the $m$ for the two sets of indices $\alpha, \beta$ to have in common, then we select the remaining $m-s$ elements from the $n-m$ remaining elements that are not in $\alpha$.  That is, 
\[ |D_\alpha| = \sum_{s=1}^{m} \binom{n-m}{m-s} \binom{m}{s} = \binom{n}{m} - \binom{n-m}{m}, \qquad \mbox{for all $\alpha\in J_m$}. \qedhere\]
\end{proof}

\begin{lemma} 
We have 
\[d_1 = \binom{n}{m}\left(\binom{n}{m} - \binom{n-m}{m}\right) / m!^2. \]
\end{lemma}
\begin{proof}
Follows immediately from Lemma~\ref{count:D} and Lemma~\ref{lemma:pab}. 
\end{proof}

The expression given for $d_2$ in Equation \eqref{d2} in the statement of Theorem \ref{theorem:main} is straightforward, although it  
contains the overlap quantities $L_s(\tau),$ which can vary wildly for different patterns $\tau$, and for which we are unaware of any general explicit or asymptotic expression.
We calculate explicit upper bounds for $d_2$ in Lemma~\ref{better:d2:bounds}. 

\subsection{Proof of Theorem~\ref{consecutive patterns} and Theorem~\ref{main:mallows:theorem}}

Theorem~\ref{consecutive patterns} follows from Theorem~\ref{main:mallows:theorem} using $q = 1$ and the fact that $I_n(1) = n!$. 
Theorem~\ref{main:mallows:theorem} is a straightforward generalization of Proposition~\ref{prop:increasing}.  

\subsection{Proof of Theorem~\ref{theorem:asymptotic-general} and Theorem~\ref{theorem:asymptotic-consecutive}}

\label{asymptotic:proof}
To prove Theorem~\ref{theorem:asymptotic-general} it is sufficient that the bounds for $d_1$ and $d_2$ in Theorem \ref{theorem:main} converge to 0 as $n\to\infty$.
For $d_2$, the asymptotic analysis is not so straightforward, which is why we 
instead use the inequality in Equation~\eqref{pab}.
\begin{lemma}\label{better:d2:bounds}
We have 
\[ d_2 \leq \frac{\binom{n}{m}}{m!^2} \sum_{s=1}^{m-1} \binom{n-m}{m-s}\binom{m}{s} \ s! \leq \frac{e^{2\left(\frac{1}{12} - \frac{3}{13}\right)}}{(2\pi)^2} \left(\frac{e^2(n-m)}{j^2}\right)^m \frac{1}{m} \ e^{2\sqrt{n-m}} \log(m) . \]
Whence, for any $\epsilon \geq 0$, taking $m \geq (e\, e^{1/e}+\epsilon) \sqrt{n}$ gives $d_2 \to 0$.  
\end{lemma}
\begin{proof}
We count the number of pairs $(\alpha, \beta)$, $\alpha \in  J_m$ and $\beta \in D_\alpha$ with exactly $s$ shared elements.  
We may first choose any $2m-s$ locations among the $n$ possible choices for the patterns to occur. 
Of those $2m-s$ locations, we can choose any $m$ of them for the elements of $\alpha$.
Then, of those $m$ locations, any $s$ can also be shared with $\beta$.
Thus, for a given $s \in \{1,2,\ldots, m-1\}$, there are
\[\binom{n}{2m-s}\binom{2m-s}{m} \binom{m}{s} = \binom{n}{m} \binom{n-m}{m-s}\binom{m}{s}\]
 terms in the sum. 
Using equation~\eqref{pab}, we have 
\[ d_2 \leq \frac{\binom{n}{m}}{m!^2} \sum_{s=1}^{m-1} \binom{n-m}{m-s}\binom{m}{s} \ s! = \frac{\binom{n}{m}}{m!} \sum_{s=1}^{m-1} \frac{\binom{n-m}{m-s}}{(m-s)!}. \]
In order to handle the sum, we first recall the quantitative bounds of Robbins~\cite{Robbins}, i.e.,
\[ e^{\frac{1}{12n+1}} < \frac{n!}{(n/e)^n \sqrt{2\pi n}} < e^{\frac{1}{12n}}, \qquad \mbox{\emph{for all $n\geq 1$}}. \]
Again we emphasize that this inequality holds for all $n \geq 1$, which allows us to provide the simpler bound of
\[ d_2 \leq \lambda \exp\left( \left(\frac{1}{12}-\frac{3}{13}\right)\right) \sum_{s=1}^{m-1} \left(\frac{e^2(n-m)}{(m-s)^2}\right)^{m-s} \frac{e^{-\frac{(m-s)^2}{n-m}}}{2\pi (m-s)}. \]
The term $\left(\frac{e^2(n-m)}{(m-s)^2}\right)^{m-s}$ is maximized when $m-s = \sqrt{n-m}$; whence
\[ d_2 \leq \frac{\lambda}{2\pi}\exp\left( \left(\frac{1}{12}-\frac{3}{13}\right)\right)  \left(e^2\right)^{\sqrt{n-m}} \sum_{s=1}^{m-1} \frac{e^{-s^2/(n-m)}}{s} \leq \frac{\lambda}{2\pi} \exp\left( \left(\frac{1}{12}-\frac{3}{13}\right)\right)  e^{2\sqrt{n-m}} \log(m). \]
Note next that
\[ \lambda \leq  \left(\frac{e^2\, n}{m^2}\right)^{m} \frac{e^{-\frac{m^2}{n}}}{2\pi\, m}, \]
so that for $\eta > 0$ and $m \sim (e + \eta) \sqrt{n}$, we have
\[ \lambda \leq \left(1 + \frac{\eta}{e}\right)^{2m} \frac{e^{-(e+\eta)^2}}{2\pi\, m} \]
and
\[ d_2 \leq \left( \left(1+\frac{\eta}{e}\right)\left(e^{\frac{1}{e+\eta}}\right)\right)^{2m} \frac{\log m}{2\pi m} \leq \left(\frac{e^{1/e}}{1+\frac{\eta}{e}}\right)^{2m} \frac{\log m}{2\pi m}. \]
Letting $\eta = e(e^{1/e}-1) +\epsilon$ for any $\epsilon \geq 0$, we conclude that taking $m \geq (e\, e^{1/e}+\epsilon)\sqrt{n}$ implies $d_2 \to 0$. 
\end{proof}

We now compute an explicit upper bound for $\overline{d}_2$ which establishes Theorem~\ref{theorem:asymptotic-consecutive}.
\begin{lemma}
\[ \overline{d}_2 \leq  \sum_{s=1}^{m-1} (n-2m+s)\frac{s!}{(m)!^2} \sim \frac{n}{m!} \frac{1}{m}. \]
Taking any fixed positive $t$, $m \geq \Gamma^{(-1)}(n/t)-1$ and $m \sim \Gamma^{(-1)}(n/t)-1$, we have $\overline{d}_2 \leq \frac{t}{m} \to 0$ as $n$ tends to infinity. 
\end{lemma}
\begin{proof}
It is easy to see that
\[ \sum_{s=1}^{m-1} \frac{s!}{m!} = \frac{1}{m}\left(1 + O(m^{-1})\right), \]
and also
\[ \sum_{s=1}^{m-1} \frac{s\, s!}{m!} = O(1). \] 
Using equation~\eqref{pab}, the result immediately follows. 
\end{proof}

For a more explicit form of the growth of $m$, we 
 define  $\psi(x) = \Gamma'(x) / \Gamma(x)$, the digamma function, as the logarithmic derivative of the gamma function, and denote by $k_0$ the smallest positive root of $\psi(x)$, i.e., $k_0 = 1.46163\ldots$. 
Also, let $c = e^{-1}\sqrt{2\pi} - \Gamma(k_0) = 0.036534\ldots$ and denote by $W(x)$ the Lambert W function, i.e., the solution to $x = W(x) e^{W(x)}$.  
Finally, let $L(x) := \log((x+c) / 2\pi)$.

\begin{lemma}[\cite{Cantrell}]
As $x$ tends to infinity, we have
\[\Gamma^{(-1)}(x) \sim \frac{L(x)}{W( L(x) / e)} + \frac{1}{2} \sim \frac{\log(x)}{\log\log(x) - \log\log\log(x)}. \]
\end{lemma}

\begin{lemma}\label{lemma:cantrell}
Suppose $t>0$ is some fixed constant and $m = \lceil\Gamma^{(-1)}(n/t)-1\rceil$.  Then
\begin{align*}
\overline{\lambda} & = \frac{n-m}{m!} \to t.
 \end{align*}
\end{lemma}

\begin{rmk}
We must be slightly careful when specifying the length of the pattern $m$ in Lemma~\ref{lemma:cantrell}, since in general $\Gamma^{(-1)}(n/t)-1$ will not be an integer.  However, as long as $m$ always exceeds this value, which we have ensured by setting it equal to the smallest integer exceeding it, then the asymptotic expressions still hold. 
\end{rmk}

\section{Acknowledgements}
The authors would like to acknowledge helpful comments from Brendan McKay, Dan Romik, Igor Rivin, Jim Pitman, Richard Arratia, Igor Pak, Michael Albert, Vince Vatter, and Doron Zeilberger.


\bibliographystyle{acm}
\bibliography{patterns}

\end{document}